\documentclass[a4paper, 10pt]{amsart}
\usepackage{amssymb,amsthm,amstext,amsmath,amscd,latexsym,amsfonts,amscd}
\usepackage[all]{xy}
\newtheorem{theorem}{Theorem}[section]
\newtheorem{lemma}[theorem]{Lemma}

\newtheorem{corollary}[theorem]{Corollary}

\theoremstyle{definition}
\newtheorem{definition}[theorem]{Definition}
\newtheorem{example}[theorem]{Example}

\theoremstyle{remark}
\newtheorem{remark}[theorem]{Remark}
\newcommand{\C}{\mathcal C}
\newcommand{\D}{\mathcal D}
\newcommand{\F}{\mathcal F}
\newcommand{\HH}{\mathcal H}
\newcommand{\SE}{\mathcal S}
\newcommand{\T}{\mathcal T}
\newcommand{\U}{\mathcal U}
\newcommand{\V}{\mathcal V}
\newcommand{\X}{\mathcal X}
\newcommand{\Y}{\mathcal Y}

\newcommand{\m}{\mathfrak m}
\newcommand{\p}{\mathfrak p}

\newcommand{\G}{\varGamma }

\newcommand{\Ass}{\mathrm {Ass}}
\newcommand{\Hom}{\mathrm {Hom}}
\newcommand{\Spec}{\mathrm {Spec}}
\newcommand{\Supp}{\mathrm {Supp}}

\newcommand{\RMod}{R\text{-}\mathrm{Mod}}

\newcommand{\sh}{\mathchar`-}
\makeatletter
\@namedef{subjclassname@2020}{\textup{2020} Mathematics Subject Classification}
\makeatother
\begin{document}
\title[Mutations of torsion theories]{Mutations of ordinary torsion theories and generalized torsion theories connected by a Serre subcategory} %[short title]{main title}

%=Information for first author================================================
\author{Takeshi Yoshizawa} 
\address{National Institute of Technology (Kosen), Toyota College, 2-1 Eiseicho, Toyota, Aichi, Japan, 471-8525}
\email{tyoshiza@toyota-ct.ac.jp}
 %13D30 Torsion theory for commutative rings %13C60=Module categories
%\thanks{This work was supported by JSPS KAKENHI Grant Number JP23K03060.}
%===================================================================
%
%
%
%
%
%
%=abstract=============================================================
\begin{abstract}%74
Understanding how torsion theories are described and constructed is crucial to the study of torsion theory. 
Mutations of torsion theories have been studied as a method of constructing another torsion theory from a given one. 
We have already obtained how to mutate ordinary torsion theories into generalized torsion theories associated with a Serre subcategory. 
The paper investigates when the generalized torsion theories give ordinary torsion theories and when ordinary torsion theories provide each other. 
\end{abstract}
%======================================================= =============
%
%
%
\maketitle
%
%
%
%=introduction============================================================
\section{Introduction}
Dickson \cite{D-1966} introduced the concept of torsion theory in an abelian category. 
Owing to its applicability, the torsion theory is highly valuable in various fields, including ring theory, representation theory, and category theory. 
For instance, in representation theory, we investigate the behavior of the category of modules using the Hasse quiver of the complete lattice of torsion classes. 

The study of how to describe and construct torsion theories is crucial. 
An indispensable result is Gabriel's classification of hereditary torsion theories in the category of modules with specialization closed subsets of the spectrum of ring, as presented in his work \cite{G-1962}. (See also \cite[Proposition 2.3]{HPST-2014} and \cite[Corollary 2.7]{K-2008}.) 

On the other hand, it is also important to investigate how to construct another torsion theory or t-structure from a given one in the category of modules or a triangulated category.

%3
%========================================
%= brick labeling
%========================================
As a specific example, a torsion theory can mutate into another torsion theory using brick modules. 
Let us consider the complete lattice $\mathrm{tors}\, A$ of torsion classes in the category of finitely generated $A$-modules over a finite-dimensional $k$-algebra $A$ for a field $k$ and the {\it heart} $\HH_{[\U, \T]}=\U^{\perp} \cap \T$ of the interval $[\U, \T]$ in $\mathrm{tors}\, A$. 
\cite[Theorems 3.3 and 3.4]{DIRRT-2023} deduce that there exists an arrow $q : \T \to \U$ in the Hasse quiver $\mathrm{Hasse}(\mathrm{tors}\, A)$ if and only if there is a unique brick $S_{q} \in \mathrm{brick}\, \HH_{[\U, \T]}$ up to isomorphism. 
In this case, one has the equality $\HH_{[\U, \T]}=\mathrm{Filt}\, S_{q}$. 
This brick $S_{q}$ is called the {\it brick label} of the arrow $q$. 
The brick labeling plays a useful role in the computation of $\mathrm{Hasse}(\mathrm{tors} A)$. 
Indeed, one has the equality $\T=\U*\mathcal{H}_{[\U, \T]}$, in other words, the torsion class $\U$ mutates into another torsion class $\T$ by the brick $S_{q}$. (See also the dual version of \cite[Lemma 2.8]{ES-2022}.)

%3
%========================================
%= HRS-tilt
%========================================
As another example, let $\HH_{(\X,\Y)}=\X[-1] \cap \Y$ be the {\it heart} of $t$-structure $(\X, \Y)$ in a triangulated category. 
Happel, Reiten, and Smal$\text{\o}$ \cite{HRS-1996} introduced methods to construct another $t$-structure from the given $t$-structure $(\X, \Y)$ using a torsion theory $(\U, \V)$ in the heart $\HH_{(\X, \Y)}$. 
To be more precise, the $t$-structure $(\X, \Y)$ mutates into the {\it left HRS-tilt} $(\X[1]*\U[1], \V[1]*\Y)$ and the {\it right HRS-tilt} $(\X*\U, \V*(\Y[-1]))$ at the pair $(\U, \V)$.

%10
%========================================
%= $\SE$-torsion theory
%========================================
There has been a growing interest in how a given torsion theory mutates into another one. 
Therefore, the paper focuses on the mutation of torsion theories in the category $\RMod$ of $R$-modules over a commutative Noetherian ring $R$. 
As a method of mutation of torsion theory, the paper uses generalized torsion theories connected by a Serre subcategory $\SE$ of $\RMod$. 
In analogy with the definition of ordinary torsion theory, the paper \cite{Y-2020} introduced this generalized torsion theory as follows. 
A pair $(\T, \F)$ of subcategories of $\RMod$ is called an {\it $\SE$-torsion theory} in $\RMod$ if it satisfies the following three conditions: 
\begin{enumerate}%\setlength{\leftskip}{-3.5pt}
\item[(TT1)]\,  If $T \in \T$, $F \in \F$ and $\varphi \in \Hom_{R}(T, F)$, then $\varphi(T)$ lies in $\SE$. 

%------
\item[(TT2)] If an $R$-module $M$ satisfies $\varphi(M) \in \SE$ for all $F \in \F$ and all $\varphi \in \Hom_{R}(M, F)$, then $M$ lies in $\T$. 

%------
\item[(TT3)] If an $R$-module $M$ satisfies $\varphi(T) \in \SE$ for all $T \in \T$ and all $\varphi \in \Hom_{R}(T, M)$, then $M$ lies in $\F$. 
\end{enumerate} 
Furthermore, in this case, the Serre subcategory $\SE$ is called the {\it heart} of the $\SE$-torsion theory $(\T, \F)$.

%10
%========================================
%= purpose of the paper
%========================================
\cite[Proposition 5.1]{Y-2024} shows that an ordinary torsion theory $(\X, \Y)$ in $\RMod$ always provides the $\SE$-torsion theory $(\X*\SE, \SE*\Y)$, which is called the $\SE$-connection of the pair $(\X, \Y)$. 
By contrast, the purpose of this paper is to investigate when $\SE$-torsion theories give ordinary torsion theories and when ordinary ones provide each other. 
Specifically, as mutations for a given $\SE$-torsion theory $(\T, \F)$ in $\RMod$, we consider the following three pairs: (1) the {\it left $\SE$-separation} $\Psi^{L}_{\SE}((\T, \F))=(\T \cap {}^{\perp} \SE, \F)$, 
(2) the {\it right $\SE$-separation} $\Psi^{R}_{\SE}((\T, \F))=(\T, \SE^{\perp} \cap \F)$, 
(3) the {\it middle $(\SE, \U, \V)$-separation} $\Psi^{M}_{(\SE, \U, \V)}((\T, \F))=( (\T \cap {}^{\perp} \SE)*\U, \V*(\SE^{\perp} \cap \F))$ for a pair $(\U, \V)$ of subcategories of $\RMod$ that are contained in the heart $\SE$. 
Note that these three separations may not be ordinary torsion theories in $\RMod$.
Therefore, it is natural to investigate when each separation becomes an ordinary torsion theory in $\RMod$.

%10
%========================================
%= three preparations(mutation, canonical torsion theory, canonical heart)
%========================================
Let us now make three preparations for stating our main results. 
Firstly, if each of the above separations is a torsion theory in $\RMod$, then the separations $\Psi^{L}_{\SE}((\T, \F))$, $\Psi^{R}_{\SE}((\T, \F))$, and $\Psi^{M}_{(\SE, \U, \V)}((\T, \F))$ are referred to as the {\it left mutation at the heart} $\SE$, the {\it right mutation at the heart} $\SE$, and the {\it middle mutation at the triple} $(\SE, \U, \V)$ of the $\SE$-torsion theory $(\T, \F)$, respectively.   
Secondly, an $\SE$-torsion theory $(\T, \F)$ is called {\it canonical} if the equality $\T*\F=\RMod$ holds. 
Finally, we say that the heart $\SE$ of an $\SE$-torsion theory $(\T, \F)$ is {\it left canonical} (respectively, {\it right canonical}) if one has the equality $(\T \cap {}^{\perp} \SE)*(\T \cap \SE)=\T$ (respectively, the equality $(\SE \cap \F)*(\SE^{\perp} \cap \F)=\F$).

%10 
%========================================
%= Main theorem
%========================================
The following main result of this paper characterizes when a given $\SE$-torsion theory $(\T, \F)$ mutates into three torsion theories in $\RMod$ as the left mutation $\Psi^{L}_{\SE}((\T, \F))$, the right mutation $\Psi^{R}_{\SE}((\T, \F))$, and the middle mutation $\Psi^{M}_{(\SE, \U, \V)}((\T, \F))$. 
Additionally, we can also verify that these three mutations generally differ from each other; see Example \ref{example-mutations}. 

\begin{theorem}[Theorems \ref{left-mutation}, \ref{right-mutation}, and \ref{middle-mutation}]\label{main-theorem}\setlength{\leftmargini}{18pt}
Let $(\T, \F)$ be an $\SE$-torsion theory and let $(\U, \V)$ be a pair of subcategories of $\RMod$ that are contained in the heart $\SE$. 
Then the following statements $(1) \sh (3)$ hold. 
\begin{enumerate}
\item[(1)] The following conditions $\mathrm{(a)}$ and $\mathrm{(b)}$ are equivalent. 
\begin{enumerate}
\item[(a)] The left $\SE$-separation $\Psi^{L}_{\SE}((\T, \F))=(\T \cap {}^{\perp} \SE, \F)$ is a torsion theory in $\RMod$. 

\item[(b)] The pair $(\T, \F)$ is canonical and the heart $\SE$ is left canonical. 
\end{enumerate}

%-----
\item[(2)] The following conditions $\mathrm{(c)}$ and $\mathrm{(d)}$ are equivalent. 
\begin{enumerate}
\item[(c)] The right $\SE$-separation $\Psi^{R}_{\SE}((\T,\F))=(\T, \SE^{\perp} \cap \F)$ is a torsion theory in $\RMod$. 

\item[(d)] The pair $(\T, \F)$ is canonical and the heart $\SE$ is right canonical. 
\end{enumerate}

%-----
\item[(3)] The following conditions $\mathrm{(e)}$ and $\mathrm{(f)}$ are equivalent. 
\begin{enumerate}
\item[(e)] The middle $(\SE, \U, \V)$-separation $\Psi^{M}_{(\SE, \U, \V)}((\T, \F))=\left( (\T \cap {}^{\perp} \SE)*\U, \V*(\SE^{\perp} \cap \F) \right)$ is a torsion theory in $\RMod$. 

\item[(f)] The following three conditions $\mathrm{(i)} \sh \mathrm{(iii})$ are satisfied: 
\begin{enumerate}
\item[(i)] The pair $(\T, \F)$ is canonical. 

\item[(ii)] The heart $\SE$ is left canonical and right canonical. 

\item[(iii)] One has the equalities $\Hom_{R}(\U, \V)=0$ and $\U*\V=\SE$. 
\end{enumerate}
\end{enumerate}

\end{enumerate}
\end{theorem}

%10
%========================================
%= A consequence of main theorem
%========================================
The second purpose of this paper is to investigate when an ordinary torsion theory in $\RMod$ mutates into another one. 
A consequence of Theorem \ref{main-theorem} characterizes when a given torsion theory $(\X, \Y)$ in $\RMod$ mutates into three types of torsion theories in $\RMod$ that are induced by the separations of the $\SE$-connection $\Phi_{\SE}((\X, \Y))=(\X*\SE, \SE*\Y)$. 

\begin{corollary}[Corollaries \ref{corollary-left-mutation}, \ref{corollary-right-mutation} and \ref{corollary-middle-mutation}]\label{main-corollary}\setlength{\leftmargini}{18pt} 
Let $\U$ and $\V$ be subcategories of $\RMod$ that are contained in a Serre subcategory $\SE$ of $\RMod$. 
For a torsion theory $(\X, \Y)$ in $\RMod$, the following statements $(1) \sh (3)$ hold. 
\begin{enumerate}
\item[(1)] The following conditions $\mathrm{(a)}$ and $\mathrm{(b)}$ are equivalent. 
\begin{enumerate}
\item[(a)] The pair $\Psi^{CL}_{\SE}((\X, \Y))=((\X*\SE)\cap{}^{\perp}\SE, \SE*\Y)$ is a torsion theory in $\RMod$. 

\item[(b)] The heart $\SE$ of the $\SE$-connection $\Phi_{\SE}((\X, \Y))=(\X*\SE, \SE*\Y)$ is left canonical. 
\end{enumerate}

%-----
\item[(2)] The following conditions $\mathrm{(c)}$ and $\mathrm{(d)}$ are equivalent. 
\begin{enumerate}
\item[(c)] The pair $\Psi^{CR}_{\SE}((\X, \Y))=(\X*\SE, \SE^{\perp} \cap (\SE*\Y))$ is a torsion theory in $\RMod$. 

\item[(d)] The heart $\SE$ of the $\SE$-connection $\Phi_{\SE}((\X, \Y))=(\X*\SE, \SE*\Y)$ is right canonical. 
\end{enumerate}

%-----
\item[(3)] The following conditions $\mathrm{(e)}$ and $\mathrm{(f)}$ are equivalent. 
\begin{enumerate}
\item[(e)] The pair $\Psi^{CM}_{(\SE, \U, \V)}((\X, \Y))= \left(\, ( (\X*\SE) \cap {}^{\perp} \SE )*\U, \V*(\SE^{\perp} \cap (\SE*\Y) ) \, \right)$ is a torsion theory in $\RMod$.  

\item[(f)] The following two conditions $\mathrm{(i)}$ and $\mathrm{(ii})$ are satisfied: 
\begin{enumerate}
\item[(i)] The heart $\SE$ of the $\SE$-connection $\Phi_{\SE}((\X, \Y))=(\X*\SE, \SE*\Y)$ is left canonical and right canonical. 

\item[(ii)] One has the equalities $\Hom_{R}(\U, \V)=0$ and $\U*\V=\SE$. 
\end{enumerate}
\end{enumerate}

\end{enumerate}
\end{corollary}

%30
%========================================
%= Organization
%========================================
This paper is organized as follows. 
Section \ref{Preliminaries} prepares notations, definitions, and fundamental facts. 
In particular, this section provides definitions of generalized torsion theory associated with a Serre subcategory and mutations of torsion theories; see Definitions \ref{Heart-torsion-theory} and \ref{definition-mutations}. 
Section \ref{section-left-right-mutation} shows Theorem \ref{main-theorem} (1), (2) and Corollary \ref{main-corollary} (1), (2); see Theorems \ref{left-mutation}, \ref{right-mutation}, and Corollaries \ref{corollary-left-mutation}, \ref{corollary-right-mutation}, respectively. 
Section \ref{section-middle-mutation} establishes Theorem \ref{main-theorem} (3) and Corollary \ref{main-corollary} (3); see Theorem \ref{middle-mutation} and Corollary \ref{corollary-middle-mutation}. 
We dedicate Section \ref{section-example} to provide an example of mutations of ordinary and generalized torsion theories; see Example \ref{example-mutations}.

%60
%==================================================================
%=
%=
%=   section : Preliminaries
%=
%=
%==================================================================
\section{Preliminaries}\label{Preliminaries}
Throughout this paper, all rings are commutative Noetherian, and all modules are unitary. 
For a ring $R$, we denote by $\RMod$ the category of $R$-modules. 
Additionally, we suppose that all subcategories of $\RMod$ are strictly full and contain the zero module. 
For convenience, we will refer to subcategories of $\RMod$ simply as {\it subcategories}.

%10
%========================================
%=  the definitions of Serre subcategory, extension subcategory, 
%=                          orthogonal subcategory, torsion theory
%========================================
First of all, we recall the concept of torsion theory introduced by Dickson \cite{D-1966} and some notions about subcategories. 

\begin{definition}\setlength{\leftmargini}{18pt}
\begin{enumerate}
\item A {\it torsion theory} in $\RMod$ is a pair $(\X, \Y)$ of subcategories satisfying the following three conditions: 
\begin{enumerate}
\item One has $\Hom_{R}(X, Y)=0$ for all $X \in \X$ and all $Y \in \Y$. 

%---
\item If an $R$-module $M$ satisfies $\Hom_{R}(M, Y)=0$ for all $Y \in \Y$, then the module $M$ lies in $\X$. 

%---
\item If an $R$-module $M$ satisfies $\Hom_{R}(X, M)=0$ for all $X \in \X$, then the module $M$ lies in $\Y$. 
\end{enumerate}
Additionally, the part $\X$ (respectively, $\Y$) is called a {\it torsion part} (respectively, {\it torsion-free part}). 

%-----
\item A {\it Serre subcategory} of $\RMod$ is a subcategory that is closed under submodules, quotient modules, and extension modules. 
{\it Unless otherwise stated, the symbol $\lq \lq \SE$" will always mean a Serre subcategory of $\RMod$.} 

%-----
\item We define the {\it extension subcategory} $\C*\D$ for the subcategories $\C$ and $\D$ as the subcategory consisting of the $R$-modules $M$ with a short exact sequence $0 \to C \to M \to D \to 0$ such that $C \in \C$ and $D \in \D$. 

%-----
\item Let $\C$ be a subcategory. 
The {\it left orthogonal subcategory} ${}^{\perp} \C$ (respectively, the {\it right orthogonal subcategory} $\C^{\perp}$) is the subcategory consisting of the $R$-modules $M$ with $\Hom_{R}(M, C)=0$ (respectively, $\Hom_{R}(C, M)=0$) for all $C \in \C$. 
\end{enumerate}
\end{definition}

%10
%========================================
%=  Remark of the canonical short exact sequence
%========================================
\begin{remark}\label{basic-remark-ordinal}\setlength{\leftmargini}{18pt}
\cite[Page 157, Exercise 8]{S-1975} (See also \cite[Chapter VI, Propositions 2.1 and 2.2]{S-1975}) states that a pair $(\X, \Y)$ of subcategories is a torsion theory in $\RMod$ if and only if it satisfies the following three conditions: 
\begin{enumerate}
\item One has $\X \cap \Y=\{0\}$.

\item The subcategory $\X$ is closed under quotient modules, and the subcategory $\Y$ is closed under submodules. 

\item One has $\X*\Y=\RMod$. 
In other words, each $R$-module $M$ has the {\it canonical short exact sequence} $0 \to X \to M \to Y \to 0$ with $X \in \X$ and $Y \in \Y$. 
\end{enumerate}
\end{remark}

%30
%========================================
%=  the definitions of torsion theory connected by a Serre subcategory
%=======================================
Next, in analogy with the definition of ordinary torsion theory, the paper \cite{Y-2020} defined a generalized torsion theory associated with a Serre subcategory. (See also \cite{Y-2024}.) 

\begin{definition}\setlength{\leftmargini}{18pt}\label{Heart-torsion-theory}
Let $\SE$ be a Serre subcategory of $\RMod$. 
\begin{enumerate}
\item A {\it torsion theory connected by $\SE$} in $\RMod$, which will be simply called an {\it $\SE$-torsion theory}, is a pair $(\T, \F)$ of subcategories satisfying the following three conditions:
\begin{enumerate}\setlength{\leftskip}{4mm}
\item[(TT1)]\,  If $T \in \T$, $F \in \F$ and $\varphi \in \Hom_{R}(T, F)$, then the module $\varphi(T)$ lies in $\SE$. 

%------
\item[(TT2)] If an $R$-module $M$ satisfies $\varphi(M) \in \SE$ for all $F \in \F$ and all $\varphi \in \Hom_{R}(M, F)$, then the module $M$ lies in $\T$. 

%------
\item[(TT3)] If an $R$-module $M$ satisfies $\varphi(T) \in \SE$ for all $T \in \T$ and all $\varphi \in \Hom_{R}(T, M)$, then the module $M$ lies in $\F$. 
\end{enumerate} 
%-----
\item Let $(\T, \F)$ be an $\SE$-torsion theory in $\RMod$. 
We refer to $\T$ (respectively, $\F$) as an {\it $\SE$-torsion part} (respectively, {\it $\SE$-torsion-free part}). 
Furthermore, the subcategory $\SE$ is called the {\it heart} of the pair $(\T, \F)$. 
\end{enumerate}
\end{definition}

%10
%========================================
%= basic properties
%========================================
The following lemma summarizes properties related to $\SE$-torsion theory that will be used frequently in this paper. 

\begin{lemma}\label{basic-lemma}\setlength{\leftmargini}{18pt}
Let $(\T, \F)$ be an $\SE$-torsion theory. 
Then the following statements hold. 
\begin{enumerate}
\item $\T$ has the property of right $\SE$-extension closed, namely one has the equality $\T=\T*\SE$. 

\item $\F$ has the property of left $\SE$-extension closed, namely one has the equality $\F=\SE*\F$. 

\item One has the equality $\T \cap \F=\SE$. 
In particular, there are the inclusion relations $\T \supseteq \SE$ , $\F \supseteq \SE$, and the equalities $\T \cap \SE=\SE=\SE \cap \F$. 

\item One has $(\T \cap {}^{\perp} \SE)*(\T \cap \SE) \subseteq \T *\SE =\T$ and $(\SE \cap \F)*(\SE^{\perp} \cap \F) \subseteq \SE*\F=\F$. %(¦ not need $\T, \F$  ext-closed)
\end{enumerate}
\end{lemma}

\begin{proof}
\cite[Lemma 3.3]{Y-2022} implies statements (1) and (2). 
Statement (3) follows from \cite[Proposition 3.4]{Y-2020}. 
Statement (4) obviously holds by statements (1) and (2). 
\end{proof}

%30
%========================================
%=  the definition of canonical
%========================================
Ordinary torsion theories provide canonical short exact sequences for each $R$-module by Remark \ref{basic-remark-ordinal} (3). 
However, we do not assume that generalized torsion theories will always give ones. 
Therefore, as an analogy to the fact that ordinary torsion theories give such sequences, it is natural to define the notion of canonical for generalized torsion theories as follows. 

\begin{definition}\setlength{\leftmargini}{18pt}
Let $(\T, \F)$ be an $\SE$-torsion theory. 
\begin{enumerate}
\item The pair $(\T, \F)$ is called {\it canonical} if one has $\T*\F=\RMod$, namely, any $R$-module $M$ has a short exact sequence $0 \to T \to M \to F \to 0$ with $T \in \T$ and $F \in \F$. 

\item The heart $\SE$ is called {\it left canonical} for the pair $(\T, \F)$ if one has $(\T \cap {}^{\perp} \SE)*(\T \cap \SE)=\T$, namely, any $R$-module $T \in \T$ has a short exact sequence $0 \to X \to T \to Y \to 0$ with $X \in \T \cap {}^{\perp} \SE$ and $Y \in \T \cap \SE$. 

\item The heart $\SE$ is called {\it right canonical} for the pair $(\T, \F)$ if one has $(\SE \cap \F)*(\SE^{\perp} \cap \F)=\F$, namely, any $R$-module $F \in \F$ has a short exact sequence $0 \to X \to F \to Y \to 0$ with $X \in \SE \cap \F$ and $Y \in \SE^{\perp} \cap \F$. 
\end{enumerate}
\end{definition}

%30
%========================================
%=  the definition of mutations
%========================================
Finally, we define the mutations of ordinary and generalized torsion theories. 
\cite[Proposition 5.1]{Y-2024} states that an ordinary torsion theory $(\X, \Y)$ in $\RMod$ always gives the $\SE$-torsion theory $(\X*\SE, \SE*\Y)$, which will be called the $\SE$-connection in the definition below. 
The purpose of this paper is to investigate when $\SE$-torsion theories provide ordinary torsion theories and when ordinary ones provide each other, as the following mutations. 

\begin{definition}\setlength{\leftmargini}{18pt}\label{definition-mutations}
Let $(\T, \F)$ be an $\SE$-torsion theory, $(\X, \Y)$ be a torsion theory in $\RMod$, and $(\U, \V)$ be a pair of subcategories that are contained in the heart $\SE$. 
\begin{enumerate}
\item The $\SE$-torsion theory $\Phi_{\SE}((\X, \Y))=(\X*\SE, \SE*\Y)$ is called the {\it $\SE$-connection} of $(\X, \Y)$. 

\item We denote by $\Psi^{L}_{\SE}((\T, \F))=(\T \cap {}^{\perp} \SE, \F)$ the {\it left $\SE$-separation} of $(\T, \F)$. 
We say that $\Psi^{L}_{\SE}((\T, \F))$ is the {\it left mutation at the heart $\SE$} of $(\T, \F)$ if the pair $(\T \cap {}^{\perp} \SE, \F)$ is a torsion theory in $\RMod$. 

\item We denote by $\Psi^{R}_{\SE}((\T, \F))=(\T, \SE^{\perp} \cap \F)$ the {\it right $\SE$-separation} of $(\T, \F)$. 
We say that $\Psi^{R}_{\SE}((\T, \F))$ is the {\it right mutation at the heart $\SE$} of $(\T, \F)$ if the pair $(\T, \SE^{\perp} \cap \F)$ is a torsion theory in $\RMod$. 

\item We denote by $\Psi^{M}_{(\SE, \U, \V)}((\T, \F))=( (\T \cap {}^{\perp} \SE)*\U, \V*(\SE^{\perp} \cap \F))$ the {\it middle $(\SE, \U, \V)$-separation} of $(\T, \F)$. 
We say that $\Psi^{M}_{(\SE, \U, \V)}((\T, \F))$ is the {\it middle mutation at the triple $(\SE, \U, \V)$} of $(\T, \F)$ if the pair $((\T \cap {}^{\perp} \SE)*\U, \V*(\SE^{\perp} \cap \F))$ is a torsion theory in $\RMod$. 

\item The pairs $\Psi^{CL}_{\SE}((\X, \Y))=(\Psi^{L}_{\SE} \circ \Phi_{\SE})((\X, \Y))$, $\Psi^{CR}_{\SE}((\X, \Y))=(\Psi^{R}_{\SE} \circ \Phi_{\SE})((\X, \Y))$, and $\Psi^{CM}_{(\SE, \U, \V)}((\X, \Y))=(\Psi^{M}_{(\SE, \U, \V)} \circ \Phi_{\SE})((\X, \Y))$ are respectively called the {\it left $\SE$-mutation}, the {\it right $\SE$-mutation}, and the {\it middle $(\SE, \U, \V)$-mutation} of $(\X, \Y)$ if each pair is a torsion theory in $\RMod$. 
\end{enumerate}
\end{definition}

%10
\begin{remark}\label{remark-mutations}
Let $(\T, \F)$, $(\X, \Y)$, and $(\U, \V)$ be as in Definition \ref{definition-mutations}. 
When we investigate mutations, we are interested in the case when the heart $\SE$ is not the zero subcategory, in other words, when the heart $\SE$ has a non-zero module.
Indeed, if we suppose that the heart $\SE=0$, then we have the equalities ${}^{\perp} \SE=\RMod$, $\SE^{\perp}=\RMod$, $\U=0$, and $\V=0$. 
Therefore, we have the equalities $\Phi_{\SE}((\X, \Y))=(\X, \Y)$, $\Psi^{L}_{\SE}((\T, \F))=\Psi^{R}_{\SE}((\T, \F))=\Psi^{M}_{(\SE, \U, \V)}((\T, \F))=(\T, \F)$. 
\end{remark}

%60
%========================================
%= 
%= left mutations and right mutations
%=
%========================================
\section{The left mutation and the right mutation}\label{section-left-right-mutation}
In this section we investigate when an $\SE$-torsion theory has the left mutation or right mutation at the heart $\SE$.

%10
%========================================
%= a necessary and sufficient condition of left mutation
%========================================
We begin with a necessary and sufficient condition for an $\SE$-torsion theory to have the left mutation at the heart $\SE$. 

\begin{theorem}\setlength{\leftmargini}{18pt}\label{left-mutation}
Let $(\T, \F)$ be an $\SE$-torsion theory. 
Then the following conditions are equivalent. 
\begin{enumerate}
\item The left $\SE$-separation $\Psi^{L}_{\SE}((\T, \F))=(\T \cap {}^{\perp} \SE, \F)$ is a torsion theory in $\RMod$. 

\item The pair $(\T, \F)$ is canonical and the heart $\SE$ is left canonical. 
\end{enumerate}
\end{theorem}

\begin{proof}
$(2) \Rightarrow (1)$: 
%(TT1)-----
We start by showing the equality $\Hom_{R}(\T \cap {}^{\perp} \SE, \F)=0$. 
Let us take $T \in \T \cap {}^{\perp} \SE$, $F \in \F$, and $\varphi \in \Hom_{R}(T, F)$. 
Then we will see the equality $\varphi=0$. 

Since the map $\varphi$ belongs to $\Hom_{R}(\T \cap {}^{\perp} \SE, \F) \subseteq \Hom_{R}(\T, \F)$, the module $\varphi(T)$ lies in $\SE$ by condition (TT1) for the pair $(\T, \F)$. 
We now decompose the map $\varphi$ into the following form 
\[ \begin{CD} T @>{\psi}>> \varphi(T) @>{i}>> F, \end{CD}\] 
where $i$ is the inclusion map. 
Then this map $\psi$ belongs to $\Hom_{R}(\T \cap {}^{\perp} \SE, \SE) \subseteq \Hom_{R}({}^{\perp} \SE, \SE)=0$. 
We thus achieve the equalities $\varphi= i \circ \psi=0$.

%(TT2)-----
Secondly, we suppose that an $R$-module $M$ satisfies $\Hom_{R}(M, \F)=0$. 
Then we will verify that the module $M$ lies in $\T \cap {}^{\perp}\SE$. 
For each map $\varphi \in \Hom_{R}(M, F)$ with $F \in \F$, our assumption for the module $M$ yields the equality $\varphi=0$, and we thus obtain that $\varphi(M)=0 \in \SE$. 
Therefore, condition (TT2) for the pair $(\T, \F)$ implies that the module $M$ lies in $\T$.
On the other hand, Lemma \ref{basic-lemma} (3) and the assumption for the module $M$ deduce that $\Hom_{R}(M, \SE) \subseteq \Hom_{R}(M, \F)=0$. 
Therefore, the module $M$ also lies in ${}^{\perp}\SE$.

%(TT3)-----
Finally, we take an $R$-module $M$ with $\Hom_{R}(\T \cap {}^{\perp}\SE, M)=0$, and let us show that the module $M$ lies in $\F$. 
Assumption (2) gives the equalities $\T*\F=\RMod$ and $(\T \cap {}^{\perp} \SE) *(\T \cap \SE)=\T$. 
Thus, the module $M$ has a short exact sequence $0 \to T \overset{\varphi}{\to} M \to F \to 0$ with $T \in \T$ and $F \in \F$. 
Furthermore, the module $T$ gives a short exact sequence $0 \to T_{1} \overset{\psi}{\to} T \to T_{2} \to 0$ with $T_{1} \in \T \cap {}^{\perp} \SE$ and $T_{2} \in \T \cap \SE$. 

We now note that the composition $\varphi \circ \psi$ belongs to $\Hom_{R}(\T \cap {}^{\perp} \SE, M)$. 
Then the assumption for the module $M$ implies that $\varphi \circ \psi=0$, and thus the equality $(\varphi \circ \psi)(T_{1})=0$ holds. 
The injectivity of the maps $\varphi$ and $\psi$ yields the equality $T_{1}=0$. 
Therefore, the second short exact sequence provides an isomorphism $T \cong T_{2} \in \T \cap \SE \subseteq \SE$. 
Consequently, the first short exact sequence and Lemma \ref{basic-lemma} (2) deduce that the module $M$ lies in $\SE*\F=\F$.

%5
%-----
$(1) \Rightarrow (2)$: 
By Remark \ref{basic-remark-ordinal} (3), the torsion theory $(\T \cap {}^{\perp} \SE, \F)$ provides that $\RMod =(\T \cap {}^{\perp} \SE)*\F \subseteq \T*\F \subseteq \RMod$. 
This implies the equality $\T*\F=\RMod$. 
In other words, the pair $(\T, \F)$ is canonical. 

Next, to prove that the heart $\SE$ is left canonical, we only have to show the relation $ (\T \cap {}^{\perp} \SE)*(\T \cap \SE) \supseteq \T$ by Lemma \ref{basic-lemma} (4). 
 Let us take an $R$-module $T \in \T$. 
Remark \ref{basic-remark-ordinal} (3) implies that the torsion theory $(\T \cap {}^{\perp} \SE, \F)$ satisfies the equality $(\T \cap {}^{\perp} \SE)*\F=\RMod$. 
Then there exists a short exact sequence $0 \to T' \to T \overset{\varphi}{\to} F \to 0$ with $T' \in \T \cap {}^{\perp} \SE$ and $F \in \F$. 
Since the map $\varphi$ belongs to $\Hom_{R}(\T, \F)$, the module $F=\varphi(T)$ lies in $\SE=\T \cap \SE$ by condition (TT1) for the pair $(\T, \F)$ and Lemma \ref{basic-lemma} (3). 
Therefore, the above short exact sequence implies that the module $T$ lies in $(\T \cap {}^{\perp} \SE)*(\T \cap \SE)$. 
This completes the proof of our claim. 
\end{proof}

%10
%========================================
%= corollary of the left mutation
%========================================
A consequence of Theorem \ref{left-mutation} characterizes when a torsion theory $(\X, \Y)$ in $\RMod$ has the left $\SE$-mutation $\Psi^{CL}_{\SE}((\X, \Y))$, that is, when the pair $\Psi^{CL}_{\SE}((\X, \Y))$ is a torsion theory in $\RMod$. 

\begin{corollary}\setlength{\leftmargini}{18pt}\label{corollary-left-mutation}
Let $(\X, \Y)$ be a torsion theory in $\RMod$. 
For a Serre subcategory $\SE$, the following conditions are equivalent. 
\begin{enumerate}
\item The pair $(\X, \Y)$ has the left $\SE$-mutation 
\[ \Psi^{CL}_{\SE}((\X, \Y))=((\X*\SE)\cap{}^{\perp}\SE, \SE*\Y).\] 

\item The heart $\SE$ of the $\SE$-connection $\Phi_{\SE}((\X, \Y))$ is left canonical. 
\end{enumerate}
\end{corollary}

\begin{proof}
Let $\T=\X*\SE$ and $\F=\SE*\Y$. 
Then we have the equalities $\Psi^{CL}_{\SE}((\X, \Y))=(\Psi^{L}_{\SE} \circ \Phi_{\SE})((\X, \Y))=\Psi^{L}_{\SE}((\T, \F))$. 
If we can verify that the pair $(\T, \F)$ is a canonical $\SE$-torsion theory, then our statements hold by Theorem \ref{left-mutation}. 

\cite[Proposition 5.1]{Y-2024} deduces that the $\SE$-connection $\Phi_{\SE}((\X, \Y))=(\T, \F)$ is an $\SE$-torsion theory. 
Moreover, Remark \ref{basic-remark-ordinal} (3) implies that the torsion theory $(\X, \Y)$ yields the relations $\RMod =\X*\Y \subseteq (\X*\SE)*(\SE*\Y) \subseteq \RMod$. 
This means the equality $\T*\F=\RMod$. 
Consequently, the pair $(\T, \F)$ is canonical. 
\end{proof}

%30
%========================================
%= a necessary and sufficient condition of right mutation
%========================================
Next, we investigate when an $\SE$-torsion theory has the right mutation at the heart $\SE$.

\begin{theorem}\setlength{\leftmargini}{18pt}\label{right-mutation}
Let $(\T, \F)$ be an $\SE$-torsion theory. 
Then the following conditions are equivalent. 
\begin{enumerate}
\item The right $\SE$-separation $\Psi^{R}_{\SE}((\T,\F))=(\T, \SE^{\perp} \cap \F)$ is a torsion theory in $\RMod$. 

\item The pair $(\T, \F)$ is canonical and the heart $\SE$ is right canonical. 
\end{enumerate}
\end{theorem}

\begin{proof}
$(2) \Rightarrow (1)$: 
%(TT1)-----
First of all, let us show the equality $\Hom_{R}(\T, \SE^{\perp} \cap \F)=0$. 
Namely, we take a map $\varphi \in \Hom_{R}(T, F)$ with $T \in \T$ and $F \in \SE^{\perp}\cap \F$, and we will prove the equality $\varphi=0$. 

Since the map $\varphi$ belongs to $\Hom_{R}(\T, \SE^{\perp} \cap \F) \subseteq \Hom_{R}(\T, \F)$, the module $\varphi(T)$ lies in $\SE$ by condition (TT1) for the pair $(\T, \F)$. 
Now let us decompose the map $\varphi$ into the following form 
\[ \begin{CD} T @>{\psi}>> \varphi(T) @>{i}>> F, \end{CD} \]
where $i$ is the inclusion map.
Then the map $i$ belongs to $\Hom_{R}(\SE, \SE^{\perp} \cap \F) \subseteq \Hom_{R}(\SE, \SE^{\perp})=0$. 
Therefore, we obtain the equality $\varphi=i \circ \psi=0$.

%(TT2)-----
Secondly, we suppose that an $R$-module $M$ satisfies $\Hom_{R}(M, \SE^{\perp} \cap \F)=0$. 
Then we will see that the module $M$ lies in $\T$. 
Assumption (2) deduces that the equalities $\T*\F=\RMod$ and $(\SE \cap \F)*(\SE^{\perp} \cap \F)=\F$. 
We thus have a short exact sequence $0 \to T \to M \overset{\varphi}{\to} F \to 0$ with $T \in \T$ and $F \in \F$. 
Moreover, the module $F$ provides a short exact sequence $0 \to F_{1} \to F \overset{\psi}{\to} F_{2} \to 0$ with $F_{1} \in \SE \cap \F$ and $F_{2} \in \SE^{\perp} \cap \F$. 

Since the composition $\psi \circ \varphi$ belongs to $\Hom_{R}(M, \SE^{\perp} \cap \F)$, the assumption for the module $M$ yields the equality $\psi \circ \varphi=0$. 
The surjectivity of the maps $\varphi$ and $\psi$ implies that the equalities $F_{2}=(\psi \circ \varphi)(M)=0$. 
Then the second short exact sequence gives an isomorphism $F \cong F_{1} \in \SE \cap \F \subseteq \SE$. 
Therefore, by Lemma \ref{basic-lemma} (1), the first short exact sequence concludes that $M$ lies in $\T*\SE=\T$.

%(TT3)-----
Finally, we take an $R$-module $M$ with $\Hom_{R}(\T, M)=0$, and let us show that the module $M$ lies in $\SE^{\perp} \cap \F$. 
Lemma \ref{basic-lemma} (3) and the assumption for the module $M$ imply that $\Hom_{R}(\SE, M) \subseteq \Hom_{R}(\T, M)=0$. 
Therefore, the module $M$ lies in $\SE^{\perp}$. 
On the other hand, for each map $\varphi \in \Hom_{R}(T, M)$ with $T \in \T$, our assumption for the module $M$ gives the equality $\varphi=0$.
This equality implies that $\varphi(T)=0 \in \SE$. 
Consequently, we can conclude that the module $M$ also lies in $\F$ by condition (TT3) for the pair $(\T, \F)$.

%------
$(1) \Rightarrow (2)$: 
By Remark \ref{basic-remark-ordinal} (3), the torsion theory $(\T, \SE^{\perp} \cap \F)$ satisfies that  $\RMod=\T*(\SE^{\perp} \cap \F) \subseteq \T*\F \subseteq \RMod$. 
Then we achieve the equality $\T*\F=\RMod$, and thus the pair $(\T, \F)$ is canonical. 

Next, let us prove that the heart $\SE$ is right canonical. 
By Lemma \ref{basic-lemma} (4), it suffices to show that $(\SE \cap \F)*(\SE^{\perp} \cap \F) \supseteq \F$. 
We suppose that an $R$-module $F$ lies in $\F$. 
Remark \ref{basic-remark-ordinal} (3) provides the equality $\T*(\SE^{\perp}\cap\F)=\RMod$ for the torsion theory $(\T, \SE^{\perp} \cap \F)$. 
This equality yields a short exact sequence $0 \to T \overset{\varphi}{\to} F \to F' \to 0$ with $T \in \T$ and $F' \in \SE^{\perp} \cap \F$. 
Since the map $\varphi$ belongs to $\Hom_{R}(\T, \F)$, the module $T \cong \varphi(T)$ lies in $\SE=\SE \cap \F$ by condition (TT1) for the pair $(\T, \F)$ and Lemma \ref{basic-lemma} (3). 
Then the above short exact sequence deduces that the module $F$ lies in $(\SE \cap \F)*(\SE^{\perp} \cap \F)$. 
This completes the proof of our inclusion relation. 
\end{proof}

%10
%========================================
%= corollary of the right mutation
%========================================
We close this section by characterizing when a torsion theory in $\RMod$ has the right $\SE$-mutation, which is a torsion theory in $\RMod$. 
The following characterization is a consequence of Theorem \ref{right-mutation}. 

\begin{corollary}\setlength{\leftmargini}{18pt}\label{corollary-right-mutation}
Let $(\X, \Y)$ be a torsion theory in $\RMod$. 
For a Serre subcategory $\SE$, the following conditions are equivalent. 
\begin{enumerate}
\item The pair $(\X, \Y)$ has the right $\SE$-mutation 
\[ \Psi^{CR}_{\SE}((\X, \Y))=(\X*\SE, \SE^{\perp} \cap (\SE*\Y)).\] 

\item The heart $\SE$ of the $\SE$-connection $\Phi_{\SE}((\X, \Y))$ is right canonical. 
\end{enumerate}
\end{corollary}

\begin{proof}
For $\T=\X*\SE$ and $\F=\SE*\Y$, we have the equalities $\Psi^{CR}_{\SE}((\X, \Y))=(\Psi^{R}_{\SE} \circ \Phi_{\SE})((\X, \Y))=\Psi^{R}_{\SE}((\T, \F))$. 
Therefore, by the same argument as the proof of Corollary \ref{corollary-left-mutation}, our statements follow from Theorem \ref{right-mutation}. 
\end{proof}

%60
%========================================
%=
%= middle mutation
%=
%========================================
\section{The middle mutation}\label{section-middle-mutation}
Our purpose here is to provide necessary and sufficient conditions for a generalized torsion theory to have the middle mutation at a triple of subcategories.

%10
We begin with the following three lemmas, which will be necessary conditions concerned with the above purpose.
The first lemma investigates whether an $\SE$-torsion theory is canonical when a middle mutation exists.

%5
%========================================
%= the pair $(\T, \F)$ is canonical
%========================================
\begin{lemma}\label{canonical-pair-lemma}
Let $(\T, \F)$ be an $\SE$-torsion theory and let $(\U, \V)$ be a pair of subcategories that are contained in the heart $\SE$. 
We suppose that the pair $\left( (\T \cap {}^{\perp} \SE)*\U, \V*(\SE^{\perp} \cap \F) \right)$ is a torsion theory in $\RMod$. 
Then the pair $(\T, \F)$ is canonical. 
\end{lemma}

\begin{proof}
Lemma \ref{basic-lemma} (1) and (2) deduce that  $(\T \cap {}^{\perp} \SE)*\U \subseteq \T*\U \subseteq \T*\SE=\T$ and $\V*(\SE^{\perp} \cap \F) \subseteq \V*\F \subseteq \SE*\F =\F$. 
Hence our assumption provides $\RMod=\left( (\T \cap {}^{\perp} \SE)*\U \right)*\left( \V*(\SE^{\perp} \cap \F) \right) \subseteq \T*\F \subseteq \RMod$ by Remark \ref{basic-remark-ordinal} (3). 
Consequently, the pair $(\T, \F)$ is canonical. 
\end{proof}

%10
%========================================
%= the heart $\SE$ is canonical
%========================================
The second lemma indicates that the existence of middle mutation for an $\SE$-torsion theory determines whether the heart is canonical. 

\begin{lemma}\label{canonical-heart-lemma}
Let $(\T, \F)$ and $(\U, \V)$ be as in Lemma \ref{canonical-pair-lemma}. 
We suppose that the pair $\left( (\T \cap {}^{\perp} \SE)*\U, \V*(\SE^{\perp} \cap \F) \right)$ is a torsion theory in $\RMod$. 
Then the heart $\SE$ is left canonical and right canonical. 
\end{lemma}

\begin{proof}
%left canonical
We start by showing that the heart $\SE$ is left canonical, namely, we will verify the equality $(\T \cap {}^{\perp} \SE) *(\T \cap \SE)=\T$. 
By Lemma \ref{basic-lemma} (4), we only need to prove that $(\T \cap {}^{\perp} \SE) *(\T \cap \SE) \supseteq \T$. 

Let us take an $R$-module $T \in \T$. 
Then, by Remark \ref{basic-remark-ordinal} (3), our assumption gives the canonical short exact sequence $0 \to X \to T \overset{\varphi}{\to} Y \to 0$ with $X \in (\T \cap {}^{\perp} \SE)*\U$ and $Y \in \V*(\SE^{\perp}\cap \F)$. 
Moreover, we can construct the following diagram of exact sequences
\[  \xymatrix{
& 0 \ar[d] & & 0 \ar[d] & & \\
& T' \ar[d] &  & V \ar[d] & & \\
0 \ar[r] & X \ar[r] \ar[d] & T \ar[r]^{\varphi} & Y \ar[r] \ar[d]^{\psi} & 0 \\
& U \ar[d] & & F \ar[d] & \\
& 0 & & 0 & & \\
  } \]
with $T' \in \T \cap {}^{\perp} \SE$, $U \in \U$, $V \in \V$, and $F \in \SE^{\perp} \cap \F$.

The composition $\psi \circ \varphi$ belongs to $\Hom_{R}(\T, \SE^{\perp} \cap \F) \subseteq \Hom_{R}(\T,  \F)$. 
Then we obtain that $F=(\psi \circ \varphi)(T) \in \SE$ by condition (TT1) for the pair $(\T, \F)$ and the surjectivity of the map $\psi \circ \varphi$. 
This shows that the identity map $\mathrm{id}_{F}$ on the module $F$ belongs to $\Hom_{R}(\SE, \SE^{\perp} \cap \F) \subseteq \Hom_{R}(\SE, \SE^{\perp})=0$. 
Since this means that $F=0$, the right column of the diagram implies an isomorphism $Y \cong V$. 

We now consider the following push-out diagram of $R$-modules with exact rows and columns: 
\[  \xymatrix{
& 0 \ar[d] & 0 \ar[d] & & & \\
& T' \ar@{=}[r] \ar[d] & T' \ar[d] & & & \\
0 \ar[r] & X \ar[r] \ar[d] & T \ar[r] \ar[d]  & V \ar[r] \ar@{=}[d]& 0 \\
0 \ar[r] & U \ar[r] \ar[d] & S \ar[r] \ar[d] & V \ar[r] & 0. \\
& 0 & 0 & & & \\
  } \]
Since the heart $\SE$ is a Serre subcategory, especially an extension closed subcategory, the bottom row of the second diagram implies that the module $S$ lies in $\U *\V \subseteq \SE *\SE =\SE$. 
Therefore, the middle column of the second diagram deduces that the module $T$ lies in $(\T \cap {}^{\perp} \SE)*\SE=(\T \cap {}^{\perp} \SE)*(\T \cap \SE)$ by Lemma \ref{basic-lemma} (3). 
Consequently, we can verify that the heart $\SE$ is left canonical.

%right canonical
Next, we will prove that the heart $\SE$ is right canonical.
To verify that the equality $(\SE \cap \F)*(\SE^{\perp} \cap \F)=\F$ holds, it suffices to show that $(\SE \cap \F) *(\SE^{\perp} \cap \F) \supseteq \F$ by Lemma \ref{basic-lemma} (4). 

We take an $R$-module $F \in \F$. 
By Remark \ref{basic-remark-ordinal} (3), our assumption provides the canonical short exact sequence $0 \to X \overset{\varphi}{\to} F \to Y \to 0$ with $X \in (\T \cap {}^{\perp} \SE)*\U$ and $Y \in \V*(\SE^{\perp}\cap \F)$. 
Then we can obtain the following diagram of exact sequences 
\[  \xymatrix{
& 0 \ar[d] & & 0 \ar[d] & & \\
& T \ar[d]^{\psi} &  & V \ar[d] & & \\
0 \ar[r] & X \ar[r]^{\varphi} \ar[d] & F \ar[r] & Y \ar[r] \ar[d] & 0 \\
& U \ar[d] & & F' \ar[d] & \\
& 0 & & 0 & & \\
  } \]
with $T \in \T \cap {}^{\perp} \SE$, $U \in \U$, $V \in \V$, and $F' \in \SE^{\perp} \cap \F$.

Since the composition $\varphi \circ \psi$ belongs to $\Hom_{R}(\T \cap {}^{\perp} \SE, \F) \subseteq \Hom_{R}(\T,  \F)$, we obtain that $T \cong (\varphi \circ \psi)(T) \in \SE$  by condition (TT1) for the pair $(\T, \F)$ and the injectivity of the map $\varphi \circ \psi$. 
Therefore, the identity map $\mathrm{id}_{T}$ on the module $T$ belongs to $\Hom_{R}( \T \cap {}^{\perp} \SE, \SE) \subseteq \Hom_{R}({}^{\perp} \SE, \SE)=0$. 
We then have the equality $T=0$, and thus the left column of the third diagram provides an isomorphism $X \cong U$. 

We make the following pull-back diagram of $R$-modules with exact rows and columns:  
\[  \xymatrix{
& & 0 \ar[d] & 0 \ar[d] & & \\
0 \ar[r] & U \ar[r] \ar@{=}[d] & S \ar[r] \ar[d] & V \ar[r] \ar[d] & 0 \\
0 \ar[r] & U \ar[r] & F \ar[r] \ar[d] & Y \ar[r] \ar[d]& 0. \\
& & F' \ar@{=}[r] \ar[d] & F' \ar[d] & & \\
& & 0 & 0 & & \\
  } \]
Since Serre subcategories are closed under extension modules, the upper row of the last diagram deduces that the module $S$ lies in $\U *\V \subseteq \SE *\SE =\SE$. 
Furthermore, the middle column of the last diagram implies that the module $F$ lies in $\SE*(\SE^{\perp} \cap \F)=(\SE \cap \F)*(\SE^{\perp}\cap \F)$ by Lemma \ref{basic-lemma} (3). 
Consequently, we can conclude that the heart $\SE$ is right canonical. 
\end{proof}

%10
%========================================
%= \U*\V=\SE, \Hom_{R}(\U, \V)=0$
%========================================
The third lemma states that if a pair $(\U, \V)$ of subcategories gives the middle mutation to an $\SE$-torsion theory, then the pair $(\U, \V)$ has to behave like a torsion theory in the heart $\SE$.
 
\begin{lemma}\label{torsion-pair-heart-lemma}
Let $(\T, \F)$ and $(\U, \V)$ be as in Lemma \ref{canonical-pair-lemma}. 
We suppose that the pair $\left( (\T \cap {}^{\perp} \SE)*\U, \V*(\SE^{\perp} \cap \F) \right)$ is a torsion theory in $\RMod$. 
Then one has the equalities $\Hom_{R}(\U, \V)=0$ and $\U*\V=\SE$. 
\end{lemma}

\begin{proof}
%\Hom_{R}(\U, \V)=0
The first equality obviously holds. 
Indeed, our assumption implies that $\Hom_{R}(\U, \V) \subseteq \Hom_{R}((\T \cap {}^{\perp} \SE)*\U, \V*(\SE^{\perp} \cap \F))=0$.

%\U*\V=\SE
Now let us prove the second equality. 
Since the heart $\SE$ is a Serre subcategory, we have $\U*\V \subseteq \SE *\SE=\SE$. 
Conversely, we take an $R$-module $S \in \SE$. 
Then our assumption implies that there exists the canonical short exact sequence $0 \to X \overset{\varphi}{\to} S \to Y \to 0$ with $X \in (\T \cap {}^{\perp} \SE)*\U$ and $Y \in \V*(\SE^{\perp} \cap \F)$ by Remark \ref{basic-remark-ordinal} (3). 
We can thus construct the following diagram of exact sequences 
\[  \xymatrix{
& 0 \ar[d] & & 0 \ar[d] & & \\
& T \ar[d]^{\psi} &  & V \ar[d] & & \\
0 \ar[r] & X \ar[r]^{\varphi} \ar[d] & S \ar[r]^{\lambda} & Y \ar[r] \ar[d]^{\nu} & 0 \\
& U \ar[d] & & F \ar[d] & \\
& 0 & & 0 & & \\
  } \]
with $T \in \T \cap {}^{\perp} \SE$, $U \in \U$, $V \in \V$, and $F \in \SE^{\perp} \cap \F$. 

The composition $\varphi \circ \psi$ belongs to $\Hom_{R}(\T \cap {}^{\perp} \SE, \SE) \subseteq \Hom_{R}({}^{\perp} \SE, \SE)=0$. 
The injectivity of the map $\varphi \circ \psi$ implies that $T \cong (\varphi \circ \psi)(T)=0$. 
Therefore, the left column of the diagram yields an isomorphism $X \cong U$. 
On the other hand, the composition $\nu \circ \lambda$ belongs to $\Hom_{R}(\SE, \SE^{\perp} \cap \F) \subseteq \Hom_{R}(\SE, \SE^{\perp})=0$. 
Since the maps $\lambda$ and $\nu$ are surjective, we obtain the equalities $F=(\nu \circ \lambda)(S)=0$. 
Then the right column of the diagram gives an isomorphism $Y \cong V$. 
Consequently, the row of the diagram provides the short exact sequence $0 \to U \to S \to V \to 0$ with $U \in U$ and $V \in \V$. 
Namely, we can conclude that the module $S$ lies in $\U*\V$. 
\end{proof}

%30
%========================================
%= the middle mutation
%========================================
Combining the necessary conditions in the above three lemmas, we can now obtain sufficient conditions for an $\SE$-torsion theory to have the middle mutation at the triple $(\SE, \U, \V)$ of subcategories. 

\begin{theorem}\setlength{\leftmargini}{18pt}\label{middle-mutation}
Let $(\T, \F)$ be an $\SE$-torsion theory and let $(\U, \V)$ be a pair of subcategories that are contained in the heart $\SE$. 
Then the following conditions $\mathrm{(1)}$ and $\mathrm{(2)}$ are equivalent. 
\begin{enumerate}
\item The middle $(\SE, \U, \V)$-separation $\Psi^{M}_{(\SE, \U, \V)}((\T, \F))=\left( (\T \cap {}^{\perp} \SE)*\U, \V*(\SE^{\perp} \cap \F) \right)$ is a torsion theory in $\RMod$. 

\item The following three conditions are satisfied: 
\begin{enumerate}
\item The pair $(\T, \F)$ is canonical. 

\item The heart $\SE$ is left canonical and right canonical. 

\item One has the equalities $\Hom_{R}(\U, \V)=0$ and $\U*\V=\SE$. 
\end{enumerate}
\end{enumerate}
\end{theorem}

\begin{proof}
For simplicity, we denote by $\X=(\T \cap {}^{\perp} \SE)*\U$ and $\Y=\V*(\SE^{\perp} \cap \F)$.

$(1) \Rightarrow (2)$: It follows from Lemmas \ref{canonical-pair-lemma}, \ref{canonical-heart-lemma}, and \ref{torsion-pair-heart-lemma}. 

$(2) \Rightarrow (1)$:
%(TT1) 
We start by showing the equality $\Hom_{R}(\X, \Y)=0$. 
Let us take a map $\varphi \in \Hom_{R}(X, Y)$ with $X \in \X$ and $Y \in \Y$. 
Then we will achieve the equality $\varphi=0$. 

From the assumption of the modules $X$ and $Y$, we can take the following diagram of exact sequences 
\[  \xymatrix{
0 \ar[r]& T \ar[r]^{\alpha}& X \ar[r]^{\beta} \ar[d]^{\varphi} & U \ar[r]& 0 \\
0 \ar[r]& V \ar[r]^{\gamma}& Y \ar[r]^{\delta}& F \ar[r]& 0 \\
  } \]
with $T \in \T \cap {}^{\perp} \SE$, $U \in \U$, $V \in \V$, and $F \in \SE^{\perp} \cap \F$. 

Since the composition $\delta \circ \varphi \circ \alpha$ belongs to $\Hom_{R}(\T \cap {}^{\perp} \SE, \SE^{\perp} \cap \F) \subseteq \Hom_{R}(\T, \F)$, the module $(\delta \circ \varphi \circ \alpha)(T)$ lies in $\SE$ by condition (TT1) for the pair $(\T, \F)$. 
We now decompose the map $\delta \circ \varphi \circ \alpha$ into the following form 
\[ \begin{CD} T@>{\eta}>> (\delta \circ \varphi \circ \alpha)(T) @>{i}>> F, \end{CD}\] 
where $i$ is the inclusion map. 
Then the inclusion map $i$ belongs to $\Hom_{R}(\SE, \SE^{\perp} \cap \F) \subseteq \Hom_{R}(\SE, \SE^{\perp})=0$. 
Hence, we obtain the equality $\delta \circ \varphi \circ \alpha=i\circ \eta=0$. 
Consequently, there exists $\lambda \in \Hom_{R}(T, V)$ such that $\varphi \circ \alpha=\gamma \circ \lambda$, and we obtain the following commutative diagram with the exact row. 
\[  \xymatrix{
 & & T \ar[d]^(0.6){\varphi \circ \alpha} \ar@{.>}[ld]_(0.4){\lambda } \ar[rd]^{0}& &  \\
0 \ar[r]& V \ar[r]^{\gamma}& Y \ar[r]^{\delta}& F \ar[r]& 0 \\
  } \]
  
The above map $\lambda$ is a morphism from the module $T \in \T \cap {}^{\perp} \SE \subseteq {}^{\perp} \SE$ to the module $V \in \V \subseteq \SE$. 
We thus have the equality $\lambda=0$. 
This implies that $\varphi \circ \alpha=\gamma \circ \lambda=0$. 
Therefore, there exists $\mu \in \Hom_{R}(U, Y)$ such that $\varphi=\mu \circ \beta$, and we have the following commutative diagram of exact sequences. 
\[  \xymatrix{
0 \ar[r]& T \ar[r]^{\alpha} \ar[rd]_{0} & X \ar[r]^{\beta} \ar[d]^(0.4){\varphi} & U \ar[r] \ar@{.>}[ld]^{\mu} & 0 \\
0 \ar[r]& V \ar[r]^{\gamma}& Y \ar[r]^{\delta}& F \ar[r]& 0 \\
  } \]

The composition $\delta \circ \mu$ is a morphism from the module $U \in \U \subseteq \SE$ to the module $F \in \SE^{\perp} \cap \F \subseteq \SE^{\perp}$. 
Hence, we obtain the equality $\delta \circ \mu=0$. 
Consequently, there exists $\nu \in \Hom_{R}(U, V)$ such that $\mu=\gamma \circ \nu$, and this gives the following commutative diagram with the exact row. 
\[  \xymatrix{
 & & U \ar[d]^(0.6){\mu} \ar@{.>}[ld]_(0.4){\nu } \ar[rd]^{0}& &  \\
0 \ar[r]& V \ar[r]^{\gamma}& Y \ar[r]^{\delta}& F \ar[r]& 0 \\
  } \]

We now note that our assumption (c) provides that $\nu \in \Hom_{R}(\U, \V)=0$. 
Consequently, we achieve the equalities $\varphi=\mu \circ \beta=\gamma \circ \nu \circ \beta=0$. 
This completes the proof of the equality $\Hom_{R}(\X , \Y)=0$.

%(TT2)
Secondly, we suppose that an $R$-module $M$ satisfies $\Hom_{R}(M, \Y)=0$. 
We claim that the module $M$ lies in $\X$. 
Since we have the equality $\T*\F=\RMod$ by assumption (a), there exists a short exact sequence $0 \to T_{1} \to M \to F_{1} \to 0$ with $T_{1} \in \T$ and $F_{1} \in \F$. 
Next, the equality $(\T \cap {}^{\perp} \SE)*(\T \cap \SE)=\T$ holds by assumption (b). 
Thus the module $T_{1}$ has a short exact sequence $0 \to T_{2} \to T_{1} \to T_{3} \to 0$ with $T_{2} \in \T \cap {}^{\perp} \SE$ and $T_{3} \in \T \cap \SE$.  
Then we can construct the following push-out diagram 
\[  \xymatrix{
& 0 \ar[d] & 0 \ar[d] & & & \\
& T_{2} \ar@{=}[r] \ar[d] & T_{2} \ar[d] & & & \\
0 \ar[r] & T_{1} \ar[r] \ar[d] & M \ar[r] \ar[d]^{\varphi}  & F_{1} \ar[r] \ar@{=}[d]& 0 \\
0 \ar[r] & T_{3} \ar[r] \ar[d] & F_{2} \ar[r] \ar[d] & F_{1} \ar[r] & 0 \\
& 0 & 0 & & & \\
  } \]
of $R$-modules with exact rows and columns. 

The bottom row of the push-out diagram implies that the module $F_{2}$ lies in $(\T\cap \SE) * \F \subseteq \SE*\F=\F$ by Lemma \ref{basic-lemma} (2). 
Since assumption (b) provides the equality $(\SE \cap \F)*(\SE^{\perp} \cap \F)=\F$, the middle column of the push-out diagram gives the following diagram of exact sequences 
\[  \xymatrix{
& & & 0 \ar[d] & & \\
& & & F_{3} \ar[d]^{\alpha} & & \\
0 \ar[r] & T_{2} \ar[r] & M \ar[r]^{\varphi} & F_{2} \ar[r] \ar[d]^{\beta} & 0 \\
& & & F_{4} \ar[d] & & \\
& & & 0 & & \\
  } \]
with $F_{3} \in \SE \cap \F$ and $F_{4} \in \SE^{\perp} \cap \F$. 
We note that the module $F_{4}$ lies in $\SE^{\perp} \cap \F \subseteq \V*(\SE^{\perp} \cap \F)=\Y$. 
Thus, our assumption for the module $M$ deduces that the composition $\beta \circ \varphi \in \Hom_{R}(M, \Y)=0$. 
Since the maps $\varphi$ and $\beta$ are surjective, the equalities $F_{4}=(\beta \circ \varphi)(M)=0$ holds. 
Then the column of the above diagram gives an isomorphism $F_{2} \cong F_{3}$. 

Let $\psi=\alpha^{-1}\circ \varphi$. 
We note that the module $F_{3}$ lies in $\SE \cap \F =\SE=\U*\V$ by Lemma \ref{basic-lemma} (3) and assumption (c). 
Hence there exists the following diagram of exact sequences 
\[  \xymatrix{
& & & 0 \ar[d] & & \\
& & & U \ar[d]^{\gamma } & & \\
0 \ar[r] & T_{2} \ar[r] & M \ar[r]^{\psi} & F_{3} \ar[r] \ar[d]^{\delta } & 0 \\
& & & V \ar[d] & & \\
& & & 0 & & \\
  } \]
with $U \in \U$ and $V \in \V$. 
Since the module $V$ lies in $\V \subseteq \V*(\SE^{\perp} \cap \F)=\Y$, our assumption for the module $M$ implies that the composition $\delta \circ \psi \in \Hom_{R}(M, \Y)=0$. 
Then the surjectivity of the maps $\psi$ and $\delta$ yields the equalities $V=(\delta \circ \psi)(M)=0$. 
Therefore, the column of the above diagram gives an isomorphism $F_{3} \cong U$. 

We can now construct the short exact sequence 
\[ \xymatrix{ 
0 \ar[r] & T_{2} \ar[r] & M \ar[r]^{\gamma^{-1} \circ \psi} & U \ar[r] & 0 
} \]
with the module $T_{2} \in \T \cap {}^{\perp} \SE$ and the module $U \in \U$. 
Consequently, we achieve that the module $M$ lies in $(\T \cap {}^{\perp} \SE)*\U=\X$, and thus the proof of our claim is completed.

%(TT3)
Finally, we suppose that an $R$-module $M$ satisfies $\Hom_{R}(\X, M)=0$. 
Then we will verify that the module $M$ lies in $\Y$. 
Assumption (a) gives the equality $\T*\F=\RMod$. 
We thus have a short exact sequence $0 \to T_{1} \to M \to F_{1} \to 0$ with $T_{1} \in \T$ and $F_{1} \in \F$. 
Additionally, we have the equality $(\SE \cap \F)*(\SE^{\perp} \cap \F)=\F$ by assumption (b). 
Therefore, the module $F_{1}$ has a short exact sequence $0 \to F_{2} \to F_{1} \to F_{3} \to 0$ with $F_{2} \in \SE \cap \F$ and $F_{3} \in \SE^{\perp} \cap \F$.  
Then we can construct the following pull-back diagram 
\[  \xymatrix{
& & 0 \ar[d] & 0 \ar[d] & & \\
0 \ar[r] & T_{1} \ar[r] \ar@{=}[d] & T_{2} \ar[r] \ar[d]^{\varphi} & F_{2} \ar[r] \ar[d] & 0 \\
0 \ar[r] & T_{1} \ar[r] & M \ar[r] \ar[d] & F_{1} \ar[r] \ar[d]& 0 \\
& & F_{3} \ar@{=}[r] \ar[d] & F_{3} \ar[d] & & \\
& & 0 & 0 & & \\
  } \]
of $R$-modules with exact rows and columns. 

The upper row of the pull-back diagram deduces that the module $T_{2}$ lies in $\T * (\SE \cap \F) \subseteq \T *\SE=\T$ by Lemma \ref{basic-lemma} (1). 
Since we have the equality $(\T \cap {}^{\perp} \SE)* (\T \cap \SE)=\T$ by assumption (b), the middle column of the pull-back diagram gives the following diagram of exact sequences 
\[  \xymatrix{
& 0 \ar[d] & & & & \\
& T_{3} \ar[d]^{\alpha} & & & & \\
0 \ar[r] & T_{2} \ar[r]^{\varphi} \ar[d]^{\beta} & M \ar[r]& F_{3} \ar[r] & 0 \\
& T_{4} \ar[d]  & & & & \\
& 0 & & & & \\
  } \]
with $T_{3} \in \T \cap {}^{\perp} \SE$ and $T_{4} \in \T \cap \SE$. 
Note that the module $T_{3}$ lies in $\T \cap {}^{\perp} \SE \subseteq (\T \cap {}^{\perp} \SE)*\U =\X$. 
Then our assumption for the module $M$ implies that the composition $\varphi \circ \alpha \in \Hom_{R}(\X, M)=0$, and we see that $(\varphi \circ \alpha)(T_{3})=0$ holds. 
The injectivity of the maps $\varphi$ and $\alpha$ deduces that $T_{3}=0$. 
Therefore, the column of the above diagram yields an isomorphism $T_{2} \cong T_{4}$. 

Lemma \ref{basic-lemma} (3) and assumption (c) imply that the module $T_{4}$ lies in $\T \cap \SE=\SE=\U *\V$. 
Then, letting $\psi=\varphi \circ \beta^{-1}$, we have the following diagram of exact sequences 
\[  \xymatrix{
& 0 \ar[d] & & & & \\
& U \ar[d]^{\gamma } & & & & \\
0 \ar[r] & T_{4} \ar[r]^{\psi} \ar[d]^{\delta } & M \ar[r] & F_{3} \ar[r] & 0 \\
& V \ar[d] & & & & \\
& 0 & & & & \\
  } \]
with $U \in \U$ and $V \in \V$. 
We note that the module $U$ lies in $\U \subseteq (\T\cap {}^{\perp} \SE)*\U=\X$. 
Thus, our assumption for the module $M$ deduces that the composition $\psi \circ \gamma \in \Hom_{R}(\X, M)=0$, and then the equality $(\psi \circ \gamma)(U)=0$ holds. 
The injectivity of the maps $\psi$ and $\gamma$ provide that $U=0$. 
Therefore, the column of the last diagram yields an isomorphism $T_{4} \cong V$. 

We now get the short exact sequence 
\[ \xymatrix{ 
0 \ar[r] & V \ar[r]^{\psi\circ \delta^{-1}}  & M \ar[r]& F_{3} \ar[r] & 0 
} \]
with the module $V \in \V$ and the module $F_{3} \in \SE^{\perp} \cap \F$. 
Consequently, we can conclude that the module $M$ lies in $\V*(\SE^{\perp} \cap \F)=\Y$. 

We now complete the proof that the pair $(\X, \Y)$ is a torsion theory in $\RMod$. 
\end{proof}

%10
%========================================
%= corollary of the middle mutation
%========================================
A consequence of Theorem \ref{middle-mutation} deduces that if there exists a certain middle mutation, then the left mutation and the right mutation are special cases of middle mutations. 

\begin{corollary}\setlength{\leftmargini}{18pt}
Let $(\T, \F)$ be an $\SE$-torsion theory. 
Then the following conditions are equivalent.
\begin{enumerate}
\item The pair $\Psi^{M}_{(\SE, 0, \SE)}((\T, \F))=\left( \T \cap {}^{\perp} \SE, \SE*(\SE^{\perp} \cap \F) \right)$ is a torsion theory in $\RMod$. 

\item The pair $\Psi^{M}_{(\SE, \SE,0)}((\T, \F))=\left( (\T \cap {}^{\perp} \SE)*\SE, \SE^{\perp} \cap \F \right)$ is a torsion theory in $\RMod$. 

\item The pair $\Psi^{L}_{\SE}((\T, \F))=\left( \T \cap {}^{\perp} \SE, \F \right)$ is a torsion theory in $\RMod$ and the heart $\SE$ is right canonical. 

\item The pair $\Psi^{R}_{\SE}((\T, \F))=\left( \T , \SE^{\perp} \cap \F \right)$ is a torsion theory in $\RMod$ and the heart $\SE$ is left canonical. 

\item The pair $(\T, \F)$ is canonical, and the heart $\SE$ is left canonical and right canonical. 
\end{enumerate}
In particular, if one of the above conditions holds, then one has the equalities $\Psi^{M}_{(\SE, 0, \SE)}((\T, \F))=\Psi^{L}_{\SE}((\T, \F))$ and $\Psi^{M}_{(\SE, \SE,0)}((\T, \F))=\Psi^{R}_{\SE}((\T, \F))$. 
\end{corollary}

\begin{proof}
We note that the pairs $(0, \SE)$ and $(\SE, 0)$ satisfy condition (c) in Theorem \ref{middle-mutation}. 
Therefore, implications (1) $\Leftrightarrow$ (5) and (2) $\Leftrightarrow$ (5) are respectively obtained as the cases of $(\U, \V)=(0, \SE)$ and $(\U, \V)=(\SE, 0)$ for Theorem \ref{middle-mutation}. 
On the other hand, Theorems \ref{left-mutation} and \ref{right-mutation} imply that implications (3) $\Leftrightarrow$ (5) $\Leftrightarrow$ (4). 

Finally, the last statement follows from the definition of left canonical heart, the definition of right canonical heart, and Lemma \ref{basic-lemma} (3). 
\end{proof}

%10
%========================================
%= corollary of the middle mutation
%========================================
For a torsion theory $(\X, \Y)$ in $\RMod$, we recall that the middle $(\SE, \U, \V)$-mutation of $(\X, \Y)$ is also a torsion theory in $\RMod$. 
We close this section by investigating the existence of the middle $(\SE, \U, \V)$-mutation of $(\X, \Y)$. 

\begin{corollary}\setlength{\leftmargini}{18pt}\label{corollary-middle-mutation}
Let $\U$ and $\V$ be subcategories that are contained in a Serre subcategory $\SE$. 
For a torsion theory $(\X, \Y)$ in $\RMod$, the following conditions $\mathrm{(1)}$ and $\mathrm{(2)}$ are equivalent. 
\begin{enumerate}
\item The pair $(\X, \Y)$ has the middle $(\SE, \U, \V)$-mutation 
\[ \Psi^{CM}_{(\SE, \U, \V)}((\X, \Y))= \left(\, ( (\X*\SE) \cap {}^{\perp} \SE )*\U, \V*(\SE^{\perp} \cap (\SE*\Y) ) \, \right).\] 

\item The following two conditions are satisfied: 
\begin{enumerate}
\item The heart $\SE$ of the $\SE$-connection $\Phi_{\SE}((\X, \Y))$ is left canonical and right canonical. 

\item One has the equalities $\Hom_{R}(\U, \V)=0$ and $\U*\V=\SE$. 
\end{enumerate}
\end{enumerate}
\end{corollary}

\begin{proof}
Let $\T=\X*\SE$ and $\F=\SE*\Y$. 
We then have the equalities 
\[ \Psi^{CM}_{(\SE, \U, \V)}((\X, \Y))=(\Psi^{M}_{(\SE, \U, \V)} \circ \Phi_{\SE})((\X, \Y))=\Psi^{M}_{(\SE, \U, \V)}((\T, \F)). \]
By the same argument as the proof of Corollary \ref{corollary-left-mutation}, our statements follow from Theorem \ref{middle-mutation}. 
 \end{proof}

%60
%========================================
%=
%= Example
%=
%========================================
\section{An example of mutations}\label{section-example}
The purpose of this section is to give an example of left, right and middle mutations. 
Although Example \ref{example-mutations} below is basic, we can provide an example of these three mutations that differ from one another.

%========================================
%= preparation
%========================================
We begin by preparing to illustrate this example. 
The following two lemmas state that Serre subcategories with a special injective module over a local ring are quite restricted to the behavior of the left orthogonal subcategories and the situation in which they are left canonical. 

We recall that a Serre subcategory $\SE$ is called {\it stable} if $\SE$ satisfies the following condition: if an $R$-module $M$ lies in $\SE$, then the injective hull $E_{R}(M)$ of $M$ also lies in $\SE$. 

\begin{lemma}\label{orthogonal-vanish-lemma}
Let $R$ be a local ring with maximal ideal $\m$. 
If the injective hull $E_{R}(R/\m)$ of the $R$-module $R/\m$ lies in a subcategory $\C$, then one has the equality $^{\perp} \C=0$. 
In particular, a stable Serre subcategory $\SE$ with a non-zero $R$-module satisfies the equality $^{\perp} \SE=0$. 
\end{lemma}

\begin{proof}
We recall that any $R$-module $M$ has the injective $R$-homomorphism
\[ \varphi_{M} : M \to \Hom_{R}(\Hom_{R}(M, E_{R}(R/\m)), E_{R}(E/\m)) \]
for which $(\varphi_{M}(x))(f)=f(x)$ for all $x \in M$ and $f \in \Hom_{R}(M, E_{R}(R/\m))$. (For example, see \cite[Remarks 10.2.2 (i)]{BS-1998}.) 
 
We now take an $R$-module $M \in {}^{\perp} \C$. 
Since our assumption says that the module $E_{R}(R/\m)$ lies in $\C$, we have the equality $\Hom_{R}(M, E_{R}(R/\m))=0$. 
Therefore, the injective map $\varphi_{M}$ gives the equality $M=0$. 

Next suppose that $\SE$ is a stable Serre subcategory with a non-zero $R$-module $S$. 
Then we can take a prime ideal $\p \in \Ass_{R}(S)$ with an injective map $R/\p \to S$. 
Additionally, we have the natural surjective map $R/\p \to R/\m$. 
We note that the stable Serre subcategory $\SE$ is closed under submodules, quotient modules, and injective hulls. 
Hence, we obtain that $E_{R}(R/\m) \in \SE$. 
Consequently, the above argument yields the equality ${}^{\perp} \SE=0$. 
\end{proof}

%10
\begin{lemma}\setlength{\leftmargini}{18pt}\label{stable-heart}
Let $R$ be a local ring with maximal ideal $\m$ and $(\T, \F)$ be an $\SE$-torsion theory. 
If the heart $\SE$ has the injective $R$-module $E_{R}(R/\m)$, then the following conditions are equivalent.
\begin{enumerate}
\item The heart $\SE$ is left canonical. 

\item One has the equality $\T=\SE$. 
\end{enumerate}
\end{lemma}

\begin{proof}
$(1) \Rightarrow (2)$: 
Since the injective $R$-module $E_{R}(R/\m)$ lies in the heart $\SE$, Lemma \ref{orthogonal-vanish-lemma} yields that the equality ${}^{\perp} \SE=0$. 
Lemma \ref{basic-lemma} (3) deduces that the left canonical heart $\SE$ provides the equalities $\T=(\T \cap {}^{\perp} \SE)*(\T \cap \SE)=0*\SE=\SE$. 

$(2) \Rightarrow (1)$: 
We note that any subcategory $\C$ has the equality $\C \cap {}^{\perp} \C=0$. 
Indeed, for each module $C \in \C \cap {}^{\perp} \C$, the identity map $\mathrm{id}_{C}$ on $C$ belongs to $\Hom_{R}({}^{\perp} \C, \C)=0$. 
We thus have $C=0$. 
 
If we suppose that the equality $\T=\SE$ holds, then we obtain the equalities $(\T \cap {}^{\perp} \SE)*(\T \cap \SE)=(\T \cap {}^{\perp} \T)*(\T \cap \T)=0*\T=\T$. 
Therefore, the heart $\SE$ is left canonical. 
\end{proof}

%30
For a specialization closed subset $W$ of $\Spec(R)$, the torsion functor $\G_{W}$ provides subcategories $\T_{W}$ and $\F_{W}$ as follows: 
\begin{align*}
\T_{W}&=\{ T \in \RMod \mid  \Supp_{R}(T) \subseteq W \}=\{ T \in \RMod \mid  \G_{W}(T)=T \}, \\
\F_{W}&=\{ F \in \RMod \mid \Ass_{R}(F) \cap W=\emptyset \}=\{ F \in \RMod \mid \G_{W}(F)=0 \}. 
\end{align*}
We recall that the pair $(\T_{W}, \F_{W})$ is a torsion theory in $\RMod$ with the equality $\T_{W}*\F_{W}=\RMod$ and that the torsion part $\T_{W}$ is a stable Serre subcategory. 
 (For example, see Remark \ref{basic-remark-ordinal}, \cite[Lemma 3.2.7 (a)]{BH-1998}, \cite[Proposition 2.3]{HPST-2014}, and \cite[Corollary 2.7]{K-2008}.)  
These facts are frequently used in Example \ref{example-mutations} below.

%10
We are now ready to construct an example of $\SE$-connection and various mutations, which is the main purpose of this section. 
In the following example, we will suppose that the specialization closed subset $W$ is a non-empty set to avoid the case where the heart $\SE=\T_{W}=0$. (See also Remark \ref{remark-mutations}.) 

\begin{example}\setlength{\leftmargini}{18pt}\label{example-mutations}
We suppose that $R$ is a local ring with maximal ideal $\m$. 
Let $V$ and $W$ be specialization closed subsets of $\Spec(R)$ such that $W$ is a non-empty set. 
We consider the torsion theory $(\X, \Y)=(\T_{V}, \F_{V})$ in $\RMod$, the stable Serre subcategory $\SE=\T_{W}$,  the extension subcategories $\T=\X*\SE=\T_{V}*\T_{W}$, and $\F=\SE*\Y=\T_{W}*\F_{V}$. 
Then the following statements hold.
\begin{enumerate}
\item The pair $(\T, \F)$ is a canonical $\SE$-torsion theory. 

\item The heart $\SE$ is right canonical, and the pair $(\T, \SE^{\perp} \cap \F)$ is a torsion theory in $\RMod$. 

\item The following conditions are equivalent: 
\begin{enumerate}
\item The pair $(\T \cap {}^{\perp} \SE, \F)$ is a torsion theory in $\RMod$.

\item The heart $\SE$ is left canonical. 

\item One has the equality $\T=\SE$. 
\end{enumerate}

\item Let $Z$ be a specialization closed subset of $\Spec(R)$. 
We consider the subcategories $\U=\SE \cap \T_{Z}$ and $\V=\SE \cap \F_{Z}$. 
Then one has the equalities $\Hom_{R}(\U, \V)=0$ and $\U*\V=\SE$. 

\item Let $\U$ and $\V$ be as in statement $(4)$. Then the following conditions are equivalent: 
\begin{enumerate}
\item The pair $\left( (\T \cap {}^{\perp} \SE)*\U, \V*(\SE^{\perp} \cap \F) \right)$ is a torsion theory in $\RMod$. 

\item The heart $\SE$ is left canonical. 
\end{enumerate}

\item We suppose that the above three specialization closed subsets satisfy that $V, Z \subseteq W$. 
Then one has the following equalities.
\begin{enumerate}
\item The $\SE$-torsion theory that is the $\SE$-connection for the torsion theory $(\X, \Y)$:
\[ (\T, \F)=\Phi_{\SE}((\X, \Y ))=(\X*\SE, \SE*\Y )=(\T_{W}, \RMod ). \]

\item The left $\SE$-mutation of $(\X, \Y)$ and the left mutation of $(\T, \F)$ at the heart $\SE$:  
\[ \Psi^{CL}_{\SE}((\X, \Y))=\Psi^{L}_{\SE}((\T, \F))=(\T \cap {}^{\perp} \SE, \F)=(0, \RMod).\]

\item The right $\SE$-mutation of $(\X, \Y)$ and the right mutation of $(\T, \F)$ at the heart $\SE$: 
\[ \Psi^{CR}_{\SE}((\X, \Y))=\Psi^{R}_{\SE}((\T, \F))=(\T, \SE^{\perp} \cap \F)=(\T_{W}, \F_{W}).\] 

\item The middle $(\SE, \U, \V)$-mutation of $(\X, \Y)$ and the middle mutation of $(\T, \F)$ at the triple $(\SE, \U, \V)$: 
\[ \Psi^{CM}_{(\SE, \U, \V)}((\X, \Y))=\Psi^{M}_{(\SE, \U, \V)}((\T, \F))=\left( (\T \cap {}^{\perp} \SE)*\U, \V*(\SE^{\perp} \cap \F) \right)=(\T_{Z}, \F_{Z}). \]
\end{enumerate}
\end{enumerate}

\vspace{5pt}
%3
%---
\noindent
(1): Since $\SE=\T_{W}$ is a Serre subcategory, \cite[Lemma 2.6]{Y-2022} deduces that the pair $(\T, \F)$ is an $\SE$-torsion theory. (See also \cite[Proposition 5.1]{Y-2024}.) 
Additionally, we have $\T*\F \supseteq \T_{V}*\F_{V}=\RMod$. 
Therefore, the pair $(\T, \F)$ is canonical.

%3
%---
\noindent
(2): To prove the equality $(\SE \cap \F)*(\SE^{\perp} \cap \F)=\F$, we only need to show $(\SE \cap \F)*(\SE^{\perp} \cap \F) \supseteq \F$ by Lemma \ref{basic-lemma} (4).

We begin by showing the equality $(\SE \cap \F)*(\SE^{\perp} \cap \F)=\T_{W}*(\F_{V} \cap \F_{W})$. 
Lemma \ref{basic-lemma} (3) gives the equalities $\SE \cap \F=\SE=\T_{W}$. 
Therefore, it remains to prove that the following equalities hold:  
\[ \SE^{\perp} \cap \F=\T_{W}^{\perp} \cap (\T_{W}*\F_{V})=\F_{W} \cap (\T_{W}*\F_{V})=\F_{V} \cap \F_{W}.\] 
The second equality follows from the well-known equality $\T_{W}^{\perp}=\F_{W}$. 
Next, it is easy to see that $\F_{W}$ and $\T_{W}*\F_{V}$ contains $\F_{V} \cap \F_{W}$. 
Conversely, we take an $R$-module $M \in \F_{W} \cap(\T_{W} * \F_{V})$. 
Then there exists a short exact sequence $0 \to T_{1} \overset{\varphi}{\to} M \to F_{1} \to 0$  with $T_{1} \in \T_{W}$ and $F_{1} \in \F_{V}$. 
Since the map $\varphi$ belongs to $\Hom_{R}(\T_{W}, \F_{W})=0$, we obtain an isomorphism $M \cong F_{1} \in \F_{V}$. 
Therefore, the module $M$ lies in $\F_{V} \cap \F_{W}$. 

Now, to prove the relation $\T_{W}*(\F_{V} \cap \F_{W}) \supseteq \F$, we take an $R$-module $F \in \F=\T_{W}*\F_{V}$. 
This module $F$ has a short exact sequence $0 \to T_{2} \to F \to F_{2} \to 0$ with $T_{2} \in \T_{W}$ and $F_{2} \in \F_{V}$. 
Additionally, we take the short exact sequence $ 0 \to \G_{W}(F_{2}) \to F_{2} \to F_{2}/\G_{W}(F_{2}) \to 0$. 
These short exact sequences provide the following pull-back diagram  
\[  \xymatrix{
& & 0 \ar[d] & 0 \ar[d] & & \\
0 \ar[r] & T_{2} \ar[r] \ar@{=}[d] & T_{3} \ar[r] \ar[d] &\G_{W}(F_{2}) \ar[r] \ar[d] & 0 \\
0 \ar[r] & T_{2} \ar[r] & F \ar[r] \ar[d] & F_{2} \ar[r] \ar[d]& 0 \\
& & F_{2}/\G_{W}(F_{2}) \ar@{=}[r] \ar[d] & F_{2}/\G_{W}(F_{2}) \ar[d] & & \\
& & 0 & 0 & & \\
  } \]
of $R$-modules with exact rows and columns.

Since the module $\G_{W}(F_{2})$ lies in $\T_{W}$, the upper row of the diagram deduces that the module $T_{3}$ lies in $\T_{W}$. 
On the other hand, \cite[Lemma 2.1]{Y-2012} implies the equality $\Ass_{R}(F_{2})=\Ass_{R}(\G_{W}(F_{2})) \cup \Ass_{R}\left( F_{2}/\G_{W}(F_{2}) \right)$. 
This equality yields $\Ass_{R}\left( F_{2}/\G_{W}(F_{2}) \right) \cap V \subseteq \Ass_{R}(F_{2}) \cap V=\emptyset$. 
Therefore, the module $F_{2}/\G_{W}(F_{2})$ lies in $\F_{V}$. 
Since this module $F_{2}/\G_{W}(F_{2})$ also lies in $\F_{W}$, the middle column of the diagram deduces that the module $F$ lies in $\T_{W}*(\F_{V} \cap \F_{W})$. 
This completes the proof of right canonical for the heart $\SE$. 

The latter assertion in statement (2) is immediate from Theorem \ref{right-mutation} because the pair $(\T, \F)$ is a canonical $\SE$-torsion theory by statement (1) and the heart $\SE$ is right canonical by the above argument. 

In addition to this proof, the following argument also provides evidence that the pair $(\T, \SE^{\perp} \cap \F)$ is a torsion theory in $\RMod$. 
Namely, let us show that each $R$-module $M$ has the canonical short exact sequence $0 \to T \to M \to F \to 0$ with $T \in \T=\T_{V}*\T_{W}$ and $F \in \SE^{\perp} \cap \F=\F_{V} \cap \F_{W}$. 
We take the short exact sequence $0 \to \G_{V}(M) \to M \to M/\G_{V}(M) \to 0$.
Let us set $N=M/\G_{V}(M)$. 
The module $N$ has the short exact sequence $0 \to \G_{W}(N) \to N \to N/\G_{W}(N) \to 0$. 
If we denote the module $N/\G_{W}(N)$ by $F$, then these short exact sequences provide the following pull-back diagram
\[  \xymatrix{
& & 0 \ar[d] & 0 \ar[d] & & \\
0 \ar[r] & \G_{V}(M) \ar[r] \ar@{=}[d] & T \ar[r] \ar[d] & \G_{W}(N) \ar[r] \ar[d] & 0 \\
0 \ar[r] & \G_{V}(M) \ar[r] & M \ar[r] \ar[d] & N \ar[r] \ar[d]& 0 \\
& & F \ar@{=}[r] \ar[d] & F \ar[d] & & \\
& & 0 & 0 & & \\
  } \]
of $R$-modules with exact rows and columns.

The upper row of the second diagram implies that the module $T$ lies in $\T_{V}*\T_{W}$. 
On the other hand, \cite[Lemma 2.1]{Y-2012} yields the equality $\Ass_{R}(N)=\Ass_{R}(\G_{W}(N)) \cup \Ass_{R}(F)$. 
Since the module $N=M/\G_{V}(M) \in \F_{V}$, we obtain that $\Ass_{R}(F) \cap V \subseteq \Ass_{R}(N) \cap V=\emptyset$. 
Therefore, the module $F$ lies in $\F_{V}$. 
Since the module $F=N/\G_{W}(N)$ also lies in $\F_{W}$, we achieve that this module $F$ lies in $\F_{V} \cap \F_{W}$.
Consequently, the middle column of the second diagram provides the desired short exact sequence for the module $M$.

%3
%---
\noindent
(3): We note that the pair $(\T, \F)$ is a canonical $\SE$-torsion theory by statement (1) and the heart $\SE=\T_{W}$ has the injective $R$-module $E_{R}(R/\m)$. 
Therefore implications (a) $\Leftrightarrow$ (b) and (b) $\Leftrightarrow$ (c) respectively follow from Theorem \ref{left-mutation} and Lemma \ref{stable-heart}. 

In addition to these implications, if one of the equivalent conditions (a)-(c) holds, then we can give the following equality 
\[ (\T \cap {}^{\perp} \SE, \F)=(0, \RMod).\]  
Indeed, the equality $\T=\SE$ in condition (c) and the proof of Lemma \ref{stable-heart} yield the equalities $\T \cap {}^{\perp} \SE=\T \cap {}^{\perp} \T=0$. 
On the other hand, the equality $\T=\SE$ also deduces that $\F=\SE*\Y=\T*\F_{V}=(\T_{V}*\T_{W})*\F_{V} \supseteq \T_{V}*\F_{V}=\RMod$. 
Therefore, we have the equality $\F=\RMod$.

%3
%---
\noindent
(4): The torsion theory $(\T_{Z}, \F_{Z})$ in $\RMod$ implies that 
\[ \Hom_{R}(\U, \V)=\Hom_{R}(\SE \cap \T_{Z}, \SE \cap \F_{Z}) \subseteq \Hom_{R}(\T_{Z}, \F_{Z})=0. \]  

On the other hand, each $R$-module $S \in \SE$ has the short exact sequence $0 \to \G_{Z}(S) \to S \to S/\G_{Z}(S) \to 0$. 
Since $\SE$ is a Serre subcategory, we obtain that $\G_{Z}(S) \in \SE \cap \T_{Z}$ and $S/\G_{Z}(S) \in \SE \cap \F_{Z}$. 
Consequently, we can achieve the equality $(\SE \cap \T_{Z})*(\SE \cap \F_{Z})=\SE$.

%3
%-----
\noindent
(5): Our assertion follows from statements (1), (2), (4) and Theorem \ref{middle-mutation}. 

Moreover, if one of the equivalent conditions (a) and (b) holds, then we can also see that the following equalities hold: 
\[ (\T \cap {}^{\perp} \SE)*\U= \T_{W} \cap \T_{Z} \ \text{ and } \ \V*(\SE^{\perp} \cap \F)=(\T_{W} \cap \F_{Z})*(\F_{V} \cap \F_{W}).\]   
Indeed, we have already obtained the equality $\T \cap {}^{\perp} \SE=0$ in the argument of statement (3). 
Thus, we obtain the equalities $(\T \cap {}^{\perp} \SE)*\U=0*\U=\T_{W} \cap \T_{Z}$. 
On the other hand, we proved the equality $\SE^{\perp} \cap \F=\F_{V} \cap \F_{W}$ in the argument of statement (2). 
Consequently, we achieve the equality $\V*(\SE^{\perp} \cap \F)=(\T_{W} \cap \F_{Z})*(\F_{V} \cap \F_{W})$.

%3
%-----
\noindent
(6): Note that our assumption $V, Z \subseteq W$ implies the relations $\T_{Z} \subseteq \T_{W}$ and $\F_{V}, \F_{Z} \supseteq \F_{W}$. 
These relations will be frequently used without mention in the following argument.

%---
\noindent 
(a): Statement (1) guarantees that the pair $(\T, \F)$ is an $\SE$-torsion theory. 
Next, we note that the assumption $V \subseteq W$ provides that $V \cup W=W$. 
Therefore, \cite[Lemma 2.2 (1)]{Y-2012} deduces the equalities 
\[ \T=\X*\SE=\T_{V}*\T_{W}=\T_{V \cup W}=\T_{W}.\] 
On the other hand, we obtain that $\F=\SE*\Y=\T_{W}*\F_{V} \supseteq \T_{W} *\F_{W}=\RMod$. 
Consequently, we have the equalities $(\T, \F)=(\X*\SE, \SE*\Y)=(\T_{W}, \RMod)$. 

%---
\noindent 
(b): The equalities $\T=\T_{W}=\SE$ follow from argument of statement (6)-(a). 
Statement (3) implies that the heart $\SE$ is left canonical and the pair $(\T \cap {}^{\perp} \SE, \F)$ is a torsion theory in $\RMod$. 
Moreover, in the argument of statement (3), we have already established the equality $(\T \cap {}^{\perp} \SE, \F)=(0, \RMod)$. 

%---
\noindent 
(c): The argument of statement (2) provides that the pair $(\T, \SE^{\perp} \cap \F)$ is a torsion theory in $\RMod$ with the equality $\SE^{\perp} \cap \F=\F_{V} \cap \F_{W}$. 
Since the equality $\T=\T_{W}$ holds by argument of statement (6)-(a), we can achieve the equalities $(\T, \SE^{\perp} \cap \F)=(\T_{W}, \F_{V} \cap \F_{W})=(\T_{W}, \F_{W})$. 

%---
\noindent 
(d): Since argument of statement (6)-(a) provides the equality $\T=\SE$, statement (3) deduces that the heart $\SE$ is left canonical. 
Therefore, statement (5) guarantees that the pair $\left( (\T \cap {}^{\perp} \SE)*\U, \V*(\SE^{\perp} \cap \F) \right)$ is a torsion theory in $\RMod$. 

Next, we recall that the argument of statement (5) gives the equalities $(\T \cap {}^{\perp} \SE)*\U=\T_{W} \cap \T_{Z}$ and $\V*(\SE^{\perp} \cap \F)=(\T_{W} \cap \F_{Z})*(\F_{V} \cap \F_{W})$. 
Moreover, we also have the equalities $\T_{W} \cap \T_{Z}=\T_{Z}$ and $(\T_{W} \cap \F_{Z})*(\F_{V} \cap \F_{W})=(\T_{W} \cap \F_{Z})*\F_{W}$. 

We now claim that the equality $(\T_{W} \cap \F_{Z})*\F_{W}=\F_{Z}$ holds. 
Indeed, since the subcategory $\F_{Z}$ is closed under extension modules, we have the equality $\F_{Z}*\F_{Z}=\F_{Z}$. 
Therefore, we have the relations $(\T_{W} \cap \F_{Z})*\F_{W} \subseteq \F_{Z}*\F_{W} \subseteq \F_{Z}*\F_{Z}=\F_{Z}$. 
Conversely, let us take an $R$-module $F \in \F_{Z}$. 
Then the module $F$ has the short exact sequence $0 \to \G_{W}(F) \to F \to F/\G_{W}(F) \to 0$. 
Since the subcategory $\F_{Z}$ is closed under submodules, the module $\G_{W}(F)$ lies in $\T_{W} \cap \F_{Z}$. 
Moreover, the module $F/\G_{W}(F)$ obviously lies in $\F_{W}$. 
Therefore, the above short exact sequence implies that the module $F$ lies in $(\T_{W} \cap \F_{Z})*\F_{W}$. 
The proof of our claim is completed. 

Consequently, we can achieve the equality 
\[ \left( (\T \cap {}^{\perp} \SE)*\U, \V*(\SE^{\perp} \cap \F) \right)=(\T_{Z}, \F_{Z}).\] 
This completes the proofs of all assertions. 
\end{example}

%60
\section*{Acknowledgments}
%The author expresses  gratitude to the referees for their kind comments and valuable suggestions. 
This work was supported by JSPS KAKENHI Grant Number JP23K03060.

%60
\vspace{7pt}

\end{document}